\documentclass{amsart}
\usepackage{amsfonts}
\usepackage{graphicx}
\usepackage{amscd}

\newtheorem{theorem}{Theorem}[section]
\theoremstyle{plain}
\newtheorem{acknowledgement}{Acknowledgement}

\newtheorem{corollary}{Corollary}[section]

\newtheorem{example}{Example}[section]

\newtheorem{lemma}{Lemma}[section]

\newtheorem{proposition}{Proposition}[section]
\newtheorem{remark}{Remark}[section]

\numberwithin{equation}{section}

\input{tcilatex}

\begin{document}
\title[Engel sub-Lorentzian structures]{Normal forms for sub-Lorentzian
metrics supported on Engel type distributions.}
\author{Marek Grochowski}
\email{m.grochowski@uksw.edu.pl}

\begin{abstract}
We construct normal forms for Lorentzian metrics on Engel distributions
under the assumption that abnormal curves are timelike future directed
Hamiltonian geodesics. Then we indicate some cases in which the abnormal
timelike future directed curve initiating at the origin is geometrically
optimal. We also give certain estimates for reachable sets from a point.
\end{abstract}

\maketitle

\section{Introduction}

\subsection{Preliminaries.}

In the series of papers \cite{g2}, \cite{g3}, \cite{g7} we studied (germs
of) contact sub-Lorentzian structures on $\mathbb{R}^{3}$. In turn, in the
series \cite{g8}, \cite{g9}, \cite{g10} some classes of non-contact
sub-Lorentzian structures on $\mathbb{R}^{3}$ were studied (in all cases the
underlying distribution is of rank $2$). The next reasonable step is to
study sub-Lorentzian structures again supported by rank $2$ distributions
but on $\mathbb{R}^{n}$, $n\geq 4$. In \ this paper we begin studies in this
direction, namely we examine the simplest such case, i.e. one supported by
the so-called Engel distribution. Before giving precise definition we will
first present basis notions and facts from the sub-Lorentzian geometry that
will be needed to state the results.

For all details and proofs the reader is referred to \cite{g6} (and to other
papers by the author; see also \cite{Vas}, \cite{markin}). Let $M$ be a
smooth manifold, and let $H$ be a smooth distribution on $M$ of constant
rank. For a point $q\in M$ and an integer $i$ let us define $H_{q}^{i}$ to
be the linear subspace in $T_{q}M$ generated by all vectors of the form $%
[X_{1},[X_{2},...,[X_{k-1},X_{k}]...]](q)$, where $X_{1},...,X_{k}$ are
smooth (local) sections of $H$ defined near $q$, and $k\leq i$. We say that $%
H$ is \textit{bracket generating} if for every $q\in M$ there exists a
positive integer $i=i(q)$ such that $H_{q}^{i}=T_{q}M$. Now, by \textit{a
sub-Lorentzian structure (or metric)} on $M$ we mean a pair $(H,g)$ made up
of a smooth bracket generating distribution $H$ of constant rank and a
smooth Lorentzian metric on $H$. A triple $(M,H,g)$ is called \textit{a
sub-Lorentzian manifold.}

Up to the end of this subsection we fix a sub-Lorentzian manifold $(M,H,g)$.
A vector $v\in H_{q}$ is called \textit{timelike} if $g(v,v)<0$, is called 
\textit{nonspacelike} if $g(v,v)\leq 0$ and $v\neq 0$, is \textit{null} if $%
g(v,v)=0$ and $v\neq 0$, finally is \textit{spacelike} if $g(v,v)>0$ or $v=0$%
. By \textit{a time orientation} of $(H,g)$ we mean a continuous timelike
vector field on $M$. Suppose that $X$ is a time orientation of $(M,H,g)$.
Then a nonspacelike $v\in H_{q}$ is said to be \textit{future directed} if $%
g(v,X(q))<0$, and is \textit{past directed} if $g(v,X(q))>0$. An absolutely
continuous curve $\gamma :[a,b]\longrightarrow M$ is called \textit{%
horizontal} if $\dot{\gamma}(t)\in H_{\gamma (t)}$ a.e. on $[a,b]$. A
horizontal curve is nonspacelike (resp. timelike, null, nonspacelike future
directed etc.) if so is $\dot{\gamma}(t)$ a.e.

Below we will need a notion of Hamiltonian geodesics. Let $\mathcal{H}%
:T^{\ast }M\longrightarrow \mathbb{R}$ be the so-called\textit{\ geodesic
(or metric) Hamiltonian} associated with our structure $(H,g)$. A global
definition of $\mathcal{H}$ is given for instance in \cite{g6}. Locally $%
\mathcal{H}$ looks as follows. Take an orthonormal basis $X_{0},...,X_{k}$
for $H$ defined on an open set $U\subset M$, where $X_{0}$ is timelike. Then
the restriction of $\mathcal{H}$ to $T^{\ast }U$ is given by $\mathcal{H}%
(q,p)=-\frac{1}{2}\left\langle p,X_{0}(q)\right\rangle ^{2}+\frac{1}{2}%
\tsum_{j=1}^{k}\left\langle p,X_{j}(q)\right\rangle ^{2}$. Denote by $%
\overrightarrow{\mathcal{H}}$ the Hamiltonian vector field corresponding to
the function $\mathcal{H}$. A horizontal curve is called \textit{a
Hamiltonian geodesic} if it can be represented in the form $\gamma (t)=\pi
\circ \lambda (t)$, where $\dot{\lambda}=\overrightarrow{\mathcal{H}}$ and $%
\pi :T^{\ast }M\longrightarrow M$ is the canonical projection. $\lambda (t)$
is called a Hamiltonian lift of $\gamma (t)$. It is immediate from the very
definition that if $\gamma :[a,b]\longrightarrow M$ is a Hamiltonian
geodesic and $\dot{\gamma}(t_{0})$ is a nonspacelike (resp. timelike, null,
nonspacelike future directed etc.) vector, then so is $\dot{\gamma}(t)$ for
every $t\in \lbrack a,b]$.

Before going further, it seems sensible to clarify why we use the word
'geodeisc'. So, first of all, if $\gamma :[a,b]\longrightarrow M$ is a
nonspacelike curve then we define its \textit{sub-Lorentzian length} by
formula 
\begin{equation*}
L(\gamma )=\int_{a}^{b}\sqrt{-g(\dot{\gamma}(t),\dot{\gamma}(t))}dt\text{.}
\end{equation*}%
Next, for an open subset $U\subset M$ and any pair of points $q_{1},q_{2}\in
U$, denote by $\Omega _{q_{1},q_{2}}^{nspc}(U)$ the set of all nonspacelike
future directed curves contained in $U$ and joining $q_{1}$ to $q_{2}$. Now
we say that a nonspacelike future directed curve $\gamma
:[a,b]\longrightarrow U$ is a \textit{maximizing} $U$\textit{-geodesic} or
simply a $U$\textit{-maximizer} if 
\begin{equation*}
L(\gamma )=\max \left\{ L(\eta ):\;\eta \in \Omega _{\gamma (a),\gamma
(b)}^{nspc}(U)\right\} \text{.}
\end{equation*}%
By a $U$\textit{-geodesic} we mean a curve in $U$ whose every sufficiently
small subarc is a $U$-maximizer (such an approach follows the ideas
elaborated in the Lorentzian case - see e.g. \cite{Beem}, \cite{ON} or \cite%
{Hawk}). It turns out \cite{g6} that for every nonspacelike Hamiltonian
geodesic $\gamma :[a,b]\longrightarrow M$ and for every $t\in (a,b)$ there
exists a neighbourhood $U$ of $\gamma (t)$ such that $U\cap \gamma $ is a $U$%
-maximizer. Note that in the Lorentzian (or Riemannian) geometry every
geodesic is Hamiltonian. It is known that in the sub-Lorentzian (or
sub-Riemannian) geometry there are maximizers (minimizers) that are not
Hamiltonian geodesics - see e.g. \cite{g3} and remark 1.1 below for examples
in the sub-Lorentzian case (and \cite{Suss}, \cite{Suss2} for the
sub-Riemannian situation).

Denote by $\Phi _{t}$ the (local) flow of the field $\overrightarrow{%
\mathcal{H}}$. For a fixed point $q_{0}\in M$ let us define $\mathcal{D}%
_{q_{0}}$ to be the set of all $\lambda \in T_{q_{0}}^{\ast }M$ such that
the curve $t\longrightarrow \Phi _{t}(\lambda )$ is defined on the whole
interval $[0,1]$. $\mathcal{D}_{q_{o}}$ is an open subset in $%
T_{q_{0}}^{\ast }M$. Now we define \textit{the exponential mapping} with the
pole at $q_{0}$ 
\begin{equation*}
\exp _{q_{0}}:\mathcal{D}_{q_{0}}\longrightarrow M\text{, \ \ }\exp
_{q_{0}}(\lambda )=\pi \circ \Phi _{1}(\lambda )\text{.}
\end{equation*}
Using properties of Hamiltonian equations it is easy to see that the
Hamiltonian geodesic with initial conditions $(q_{0},\lambda )$ can be
written as $\gamma (t)=\exp _{q_{0}}(t\lambda )$. It can also be observed
that if $\gamma (t)$ is a Hamiltonian geodesic with a Hamiltonian lift $%
\lambda (t)=\Phi _{t}(\lambda )$ then, from the definition of the geodesic
Hamiltonian (see \cite{g6} for more details), it follows that for any $v\in
H_{\gamma (t)}$ we have

\begin{equation}
g(\dot{\gamma}(t),v)=\left\langle \Phi _{t}(\lambda ),v\right\rangle \text{.}
\label{il}
\end{equation}

At the end let us recall the notion of abnormal curves (cf. e.g. \cite{Suss}%
). So an absolutely continuous curve $\lambda :[a,b]\longrightarrow T^{\ast
}M$ is called an abnormal biextremal if $\lambda ([a,b])\subset H^{\bot }$, $%
\lambda $ never intersects the zero section, and moreover $\Omega _{\lambda
(t)}(\dot{\lambda}(t),\zeta )=0$ for almost every $t\in \lbrack a,b]$ and
every $\zeta \in T_{\lambda (t)}H^{\bot }$; here $H^{\bot }$ is the
annihilator of $H$, and $\Omega $ denotes the restriction to $H^{\bot }$ of
the standard symplectic form on $T^{\ast }M$. A horizontal curve $\gamma
:[a,b]\longrightarrow M$ is said to be \textit{abnormal} if there exists an
abnormal biextremal $\lambda :[a,b]\longrightarrow T^{\ast }M$ such that $%
\gamma =\pi \circ \lambda $.

Throughout the paper we will use the following abbreviations:
\textquotedblright t.\textquotedblright\ for \textquotedblright
timelike\textquotedblright , \textquotedblright nspc.\textquotedblright\ for
\textquotedblright nonspacelike\textquotedblright , and \textquotedblright
f.d.\textquotedblright\ for \textquotedblright future
directed\textquotedblright . Moreover, unless otherwise stated, we assume
all curves and vectors to be horizontal. Thus e.g. a t.f.d. curve is a
horizontal curve whose tangent is t.f.d. a.e.

\subsection{Statement of the results.}

Let $H$ be a $\emph{rank}$ $2$ distribution of constant rank on a $4$%
-dimensional manifold $M$. We say that $H$ is an \textit{Engel (or Engel
type) distribution} if $H^{2}$ is of constant rank $3$, and $H^{3}$ is of
constant rank $4$, i.e. $H^{3}=TM$. The remarkable property of Engel
distributions is the fact that they are topologically stable, see e.g. \cite%
{Montg} (note that apart from Engel case, the only stable distributions are
rank $1$ distributions, and also contact and pseudo-contact distributions).\
On the other hand, if one slightly perturbs any given rank $2$ distribution
on a $4$-manifold it becomes Engel on an open and dense subset. All this
gives rise to the importance of Engel distributions. But Engel distributions
are important also because of another reason, namely they appear in
applications. For instance our flat case (see example below) serves as a
model for a motion of a car with a single trailer (cf. e.g. \cite{Gauth}).

Using for instance \cite{Suss} one makes sure that if $H$ is an Engel
distribution on $M$ then through each point $q\in M$ there passes exactly
one unparameterized abnormal curve. Moreover the abnormal curves are all (at
least locally) trajectories of a single smooth vector field.

Let $H$ be an Engel type distribution and let $g$ be a Lorentzian metric on $%
H$. A couple $(H,g)$ is called an \textit{Engel sub-Lorentzian structure (or
metric)} if the abnormal curves for $H$ are timelike. If moreover the
abnormal curves are, possibly after reparameterization, t.f.d. Hamiltonian
geodesics then $(H,g)$ will be called \textit{Engel sub-Lorentzian structure
of Hamiltonian type}.

\begin{example}
As a model example of an Engel sub-Lorentzian structure of Hamiltonian type
we use the following one. Let $H=Span\left\{ X,Y\right\} $ with $X=\frac{%
\partial }{\partial x}+\frac{1}{2}y\frac{\partial }{\partial z}+\frac{1}{2}%
y^{2}\frac{\partial }{\partial w}$, $Y=\frac{\partial }{\partial y}-\frac{1}{%
2}x\frac{\partial }{\partial z}-\frac{1}{2}xy\frac{\partial }{\partial w}$.
Clearly $\left[ X,Y\right] =-\frac{\partial }{\partial z}-\frac{3}{2}y\frac{%
\partial }{\partial w}$, $\left[ X,\left[ X,Y\right] \right] =0$, $\left[ Y,%
\left[ X,Y\right] \right] =-\frac{3}{2}\frac{\partial }{\partial w}$, so
indeed $H^{2}$ is of constant rank $3$ while $H^{3}$ is of constant rank $4$%
. The trajectories of $X$ are the curves $t\longrightarrow
(x_{0}+t,y_{0},z_{0}+\frac{1}{2}y_{0}t,w_{0}+\frac{1}{2}y_{0}^{2}t)$, and
these curve are easily checked to be abnormal. Now we define a metric by
declaring $X$, $Y$ to be an orthonormal basis with a time orientation $X$,
and we make sure that a curve $t\longrightarrow (x_{0}+t,y_{0},z_{0}+\frac{1%
}{2}y_{0}t,w_{0}+\frac{1}{2}y_{0}^{2}t,-1,0,0,0)$ represents a Hamiltonian
lift for the corresponding trajectory of $X$.
\end{example}

The structure just described will be called \textit{the flat Engel
sub-Lorentzian structure}. The reason for such a name is the same as in the
previous papers by the author (see for instance \cite{g8}, \cite{g9}) - any
Engel structure of Hamiltonian type may be viewed as a perturbation of the
flat structure.

\begin{remark}
It is easy to construct Engel sub-Lorentzian structures $(H,g)$ which are
not of Hamiltonian type. The idea may be taken from Sussmann who gives in 
\cite{Suss2} a simple recipe how to contract Riemannian metrics on Engel
distributions with respect to which abnormal curves are strictly abnormal,
i.e. are not Hamiltonian geodesics. The construction goes without any
changes in the case of Lorentzian metrics. So it is enough to find two
fields $V$, $W$ spanning $H$ such that (i) $V\wedge W\wedge \left[ V,W\right]
\wedge \left[ W,\left[ V,W\right] \right] \neq 0$, (ii) $\left[ V,\left[ V,W%
\right] \right] =fV+gW+h\left[ V,W\right] $, with $f,g,h$ being smooth
functions such that $f$ vanishes nowhere. Now define a Lorentzian metric on $%
H$ by declaring $V,W$ to be an orthonormal basis with time orientation $V$.
The abnormal curves for just defined structure $(H,g)$ (which all are
trajectories of $V$) are not Hamiltonian geodesics (For the convenience of
the reader we present the argument. Suppose that a trajectory $\gamma $ of $%
V $ is a t.f.d. Hamiltonian geodesic. Let $\lambda (t)$ be its Hamiltonian
lift; then by (\ref{il}) it follows that $\left\langle \lambda
(t),V\right\rangle =-1$ and $\left\langle \lambda (t),W\right\rangle =0$ for
every $t$. Now, the successive differentiations of the second equation give: 
$\left\langle \lambda (t),\left[ V,W\right] \right\rangle =0$, and $%
0=\left\langle \lambda (t),\left[ V,\left[ V,W\right] \right] \right\rangle
=f\left\langle \lambda (t),V\right\rangle +g\left\langle \lambda
(t),W\right\rangle +h\left\langle \lambda (t),\left[ V,W\right]
\right\rangle =-f$ which is a contradiction with the assumption imposed on $%
f $).
\end{remark}

The main objective of this paper is to prove the following normal form
theorem (cf. \cite{g2}, \cite{g8}, \cite{g9}, and also \cite{Agr3}).

\begin{theorem}
Let $(H,g)$ be a smooth time-oriented Engel sub-Lorentzian structure of
Hamiltonian type defined in a neighbourhood of a point $q_{0}$ on a $4$%
-manifold. Then there are coordinates $x,y,z,w$ around $q_{0}$, $%
x(q_{0})=...=w(q_{0})=0$, in which $(H,g)$ has an orthonormal frame in the
normal form 
\begin{equation}
\left. 
\begin{array}{c}
X=\dfrac{\partial }{\partial x}+y\varphi \left( y\dfrac{\partial }{\partial x%
}+x\dfrac{\partial }{\partial y}\right) +\frac{1}{2}y(1+\psi _{1})\dfrac{%
\partial }{\partial z}+\frac{1}{2}y^{2}(1+\psi _{2})\dfrac{\partial }{%
\partial w} \\ 
Y=\dfrac{\partial }{\partial y}-x\varphi \left( y\dfrac{\partial }{\partial x%
}+x\dfrac{\partial }{\partial y}\right) -\frac{1}{2}x(1+\psi _{1})\dfrac{%
\partial }{\partial z}-\frac{1}{2}xy(1+\psi _{2})\dfrac{\partial }{\partial w%
}%
\end{array}%
\right.  \label{NF}
\end{equation}%
where $\varphi $, $\psi _{1}$, $\psi _{2}$ are smooth functions satisfying $%
\psi _{1}(0,0,z,w)=\psi _{2}(0,0,0,w)=0$, and $X$ is a time orientation
whose trajectories contained in $\left\{ y=0\right\} $ are abnormal curves
for $H$.
\end{theorem}

Theorem 1.1 is a starting point to the investigation of Engel sub-Lorentzian
structures. By the way we obtain the following partial result. Let $H$ be
such a rank two bracket generating distribution on a $4$-manifold $M$ that $%
H^{2}$ is everywhere of rank $3$ (in particular, the situation $H^{3}\neq TM$
is allowed now, so one may need more Lie brackets to generate the whole
tangent space). Suppose moreover that through each point of $M$ there passes
exactly one abnormal curve and besides all abnormal curves are trajectories
of a single smooth vector field. Now let $g$ be a Lorentzian metric on $H$
such that all abnormal curves are t.f.d. Hamiltonian geodesics. Then we can
prove

\begin{proposition}
Let $(H,g)$ be a germ at $q_{0}\in M$ of a time-oriented sub-Lorentzian
structure defined above. Then, around $q_{0}$, there exist coordinates $%
x,y,z,w$, $x(q_{0})=...=w(q_{0})=0$, in which $(H,g)$ admits an orthonormal
frame in the form 
\begin{equation*}
\left. 
\begin{array}{c}
X=\dfrac{\partial }{\partial x}+y\varphi \left( y\dfrac{\partial }{\partial x%
}+x\dfrac{\partial }{\partial y}\right) +\frac{1}{2}y(1+\psi _{1})\dfrac{%
\partial }{\partial z}-yA_{2}\dfrac{\partial }{\partial w} \\ 
Y=\dfrac{\partial }{\partial y}-x\varphi \left( y\dfrac{\partial }{\partial x%
}+x\dfrac{\partial }{\partial y}\right) -\frac{1}{2}x(1+\psi _{1})\dfrac{%
\partial }{\partial z}+xA_{2}\dfrac{\partial }{\partial w}%
\end{array}%
\right.
\end{equation*}%
where $\varphi $, $\psi _{1}$, $A_{2}$ are smooth functions, $\psi
_{1}(0,0,z,w)=0$, and $X$ is a time orientation whose trajectories contained
in $\left\{ y=0\right\} $ are abnormal curves for $H$.
\end{proposition}

Using theorem 1.1, in further parts of the paper, we attempt to describe
reachable sets from the origin for Engel sub-Lorentzian structures.

If $(M,H,g)$ is a sub-Lorentzian manifold, $q_{0}$ is a fixed point in $M$,
and $U$ is a neighbourhood of $q_{0}$, then by \textit{the (future)
nonspacelike reachable set from }$q_{0}$ we mean the set of all points $q\in
U$ such that $q$ can be reached from $q_{0}$ by nspc.f.d. curve contained in 
$U$; this set will be denoted by $J^{+}(q_{0},U)$. Replacing nspc.f.d.
curves with t.f.d. and null f.d. curves we obtain the definition of \textit{%
the (future) timelike reachable set} $I^{+}(q_{0},U)$, and \textit{the
(future) null reachable set} $N^{+}(q_{0},U)$, respectively. For general $U$
we know that the three reachable sets have the same interiors (which are
nonempty) and closures relative to $U$. In order to be able to obtain more
precise results we need to impose certain assumptions on $U$. To this end
notice that if $U$ is sufficiently small then the metric $g$ can be extended
to a Lorentzian metric $\tilde{g}$ defined on a neighbourhood of $U$. Now, $%
U $ is called \textit{a normal neighbourhood of }$q_{0}$ if it is a convex
normal neighbourhood of $q_{0}$ with respect to $\tilde{g}$ (see e.g. \cite%
{ON}) and\ its closure is contained in some other convex normal
neighbourhood of $q_{0}$ with respect to $\tilde{g}$. Recall \cite{g6} in
this place that if $U$ is a normal neighbourhood of $q_{0}$ then 
\begin{equation*}
J^{+}(q_{0},U)=cl_{U}\left( int\text{ }I^{+}(q_{0},U)\right) =cl_{U}\left(
int\text{ }N^{+}(q_{0},U)\right) \text{,}
\end{equation*}%
where $cl_{U}$ is the closure with respect to $U$. The basic objects, when
studying reachable sets, are the so-called geometrically optimal curves. A
nspc.f.d. curve $\gamma :[0,T]\longrightarrow U$ is called \textit{%
geometrically optimal} if $\gamma \left( \lbrack 0,T]\right) \subset
\partial _{U}J^{+}(\gamma (0),U)$, where $\partial _{U}$ stands for the
boundary operator taken with respect to $U$. It is a standard fact that
geometrically optimal curves satisfy Pontriagin maximum principle - see eg. 
\cite{Agr}.

First of all we prove the following

\begin{proposition}
Let $(H,g)$ be given by an orthonormal frame in the normal form (\ref{NF}),
where $\varphi =\varphi (x,y,w)$, $\psi _{2}=\psi _{2}(x,y,w)$, i.e. $%
\varphi ,\psi _{2}$ do not depend on $z$. Then the abnormal curve starting
from zero (which is t.f.d) is geometrically optimal.
\end{proposition}

The proof uses the observation that lifts of geometrically optimal curves
are again geometrically optimal. Remark here that timelike abnormal curves
always satisfy necessary conditions for optimality from Pontriagin maximum
principle, and in general it is not a trivial thing to determine if a given
timelike abnormal curve is geometrically optimal or not (cf. \cite{Bonnard}
and note that timelike abnormal curves correspond to singular trajectories
of affine control systems - see \cite{g6}). Examples of timelike abnormal
curves which are not geometrically optimal can be found in \cite{g9}.

Using proposition 1.2 we come to the investigation of reachable sets. In
papers \cite{g8}, \cite{g9}, the author managed to give a precise
description of reachable sets from the origin for sub-Lorentzian structures,
where the abnormal t.f.d. curves fill a hypersurface passing through the
origin. As it is noticed above, Engel case is much harder to study since
here abnormal t.f.d. curves pass through every point. From this reason the
methods developed earlier by the author do not work (or at least require
serious modifications), and therefore we obtain only certain estimates on
reachable sets - propositions 3.2, 3.3.

\textbf{The organization of the paper.}

In section 2 we prove theorem 1.1 and proposition 1.1. In section 3 we prove
propositions 1.2, 3.2, and 3.3.

\section{Normal Forms}

Let $(H,g)$ be a time-oriented Engel sub-Lorentzian structure of Hamiltonian
type. Without loss of generality we can suppose it to be defined in a
neighbourhood $U$ of $0\in \mathbb{R}^{4}$. Throughout this section $U$ will
be supposed to be as small as it is needed to justify various statements
that are made below.

\subsection{Normal coordinates.}

Let $\tilde{X}$, $\tilde{Y}$ be an orthonormal frame for $(H,g)$ such that $%
\tilde{X}$ is a time orientation and the trajectories of $\tilde{X}$ are
exactly the abnormal curves for $H$. Such a field exists by \cite{Suss}, 
\cite{Suss2}, and we can assume that these trajectories are t.f.d.\
Hamiltonian geodesics (if it were not the case we change parameterization).
By definition of Engel structure the fields $\tilde{X}$, $\tilde{Y}$, $[%
\tilde{X},\tilde{Y}]$ are linearly independent. Choose a curve $\Gamma $
passing through the origin, and such that $\Gamma $ is transverse to $%
H^{2}=Span\left\{ \tilde{X},\tilde{Y},[\tilde{X},\tilde{Y}]\right\} $.
Denote by $g_{\tilde{X}}^{t}$ (resp. $g_{\tilde{Y}}^{t}$, $g_{[\tilde{X},%
\tilde{Y}]}^{t}$) the (local) flow of $\tilde{X}$ (resp. $\tilde{Y}$, $[%
\tilde{X},\tilde{Y}]$) defined on $U$, and let $P=\tbigcup\nolimits_{t,s}g_{%
\tilde{X}}^{t}\circ g_{[\tilde{X},\tilde{Y}]}^{s}\Gamma $; clearly $P$ is a
smooth hypersurface.

\begin{lemma}
There are coordinates $\tilde{x},\tilde{y},\tilde{z},\tilde{w}$ on $U$ such
that\newline
(i) $P=\left\{ \tilde{y}=0\right\} $, $\Gamma =\left\{ \tilde{x}=\tilde{y}=%
\tilde{z}=0\right\} $;\newline
(ii) $H_{|P}=\ker d\tilde{z}\cap \ker \tilde{w}$;\newline
(iii) $\frac{\partial }{\partial \tilde{x}}_{|P}$, $\frac{\partial }{%
\partial \tilde{y}}_{|P}$ is an orthonormal frame for $(H,g)$, and $\frac{%
\partial }{\partial \tilde{x}}_{|P}$ is a time orientation.
\end{lemma}

\begin{proof}
Let $\sigma (\tilde{w})$, $\sigma (0)=0$, be a parameterization of $\Gamma $%
. Then the coordinates we look for are given by the diffeomorphism $(\tilde{x%
},\tilde{y},\tilde{z},\tilde{w})\longrightarrow g_{\tilde{Y}}^{\tilde{y}%
}\circ g_{\tilde{X}}^{\tilde{x}}\circ g_{[\tilde{X},\tilde{Y}]}^{\tilde{z}%
}\sigma (\tilde{w})$.
\end{proof}

Since $\tilde{X}=\frac{\partial }{\partial \tilde{x}}$ on $P$, the curves $%
t\longrightarrow (t,0,z_{0},w_{0})$ (written in just constructed coordinates 
$\tilde{x},\tilde{y},\tilde{z},\tilde{w}$) are abnormal. Denote by $p_{%
\tilde{x}},p_{\tilde{y}},p_{\tilde{z}},p_{\tilde{w}}$ the dual coordinates
to $\tilde{x},\tilde{y},\tilde{z},\tilde{w}$. Then $\mathcal{H}%
_{|T_{P}^{\ast }\mathbb{R}^{4}}=-\frac{1}{2}p_{\tilde{x}}^{2}+\frac{1}{2}p_{%
\tilde{y}}^{2}$ where $T_{P}^{\ast }\mathbb{R}^{4}=\tbigcup\nolimits_{x\in
P}T_{x}^{\ast }\mathbb{R}^{4}$, and since, by assumption, the curves $%
t\longrightarrow (t,0,z_{0},w_{0})$ are abnormal t.f.d. Hamiltonian
geodesics we obtain 
\begin{equation*}
\mathcal{H}(\tilde{x},\tilde{y},\tilde{z},\tilde{w},p_{\tilde{x}},p_{\tilde{y%
}},p_{\tilde{z}},p_{\tilde{w}})=-\frac{1}{2}p_{\tilde{x}}^{2}+\frac{1}{2}p_{%
\tilde{y}}^{2}+\tilde{y}^{2}\mathcal{G}(\tilde{x},\tilde{y},\tilde{z},\tilde{%
w},p_{\tilde{x}},p_{\tilde{y}},p_{\tilde{z}},p_{\tilde{w}})
\end{equation*}%
for a smooth function $\mathcal{G}$.

Let 
\begin{equation*}
A_{\Gamma }=\left\{ \left( 0,0,\tilde{z},\tilde{w},p_{\tilde{x}},p_{\tilde{y}%
},0,0\right) :\;\left\vert \tilde{z}\right\vert ,\left\vert \tilde{w}%
\right\vert <\varepsilon \right\} \text{.}
\end{equation*}%
Consider the mapping $\mu :A_{\Gamma }\longrightarrow U$ given by 
\begin{equation*}
\mu (\tilde{z},\tilde{w},p_{\tilde{x}},p_{\tilde{y}})=\pi \circ \Phi
_{1}\left( 0,0,\tilde{z},\tilde{w},-p_{\tilde{x}},p_{\tilde{y}},0,0\right) 
\text{,}
\end{equation*}%
i.e. in terms of the exponential mapping 
\begin{equation*}
\mu (\tilde{z},\tilde{w},p_{\tilde{x}},p_{\tilde{y}})=\exp _{(0,0,\tilde{z},%
\tilde{w},)}\left( -p_{\tilde{x}},p_{\tilde{y}},0,0\right) \text{.}
\end{equation*}%
If $N$ is a sufficiently small neighbourhood of the set $\left\{ \left( 0,0,%
\tilde{z},\tilde{w},0,0,0,0\right) :\;\left\vert \tilde{z}\right\vert
,\left\vert \tilde{w}\right\vert <\varepsilon \right\} $ in $A_{\Gamma }$
then $\mu :N\longrightarrow \mu (N)$ is a diffeomorphism. Therefore we can
write $\mu (N)=U$, and now we are ready to define \textit{normal coordinates}
$x,y,z,w$ on $U$. These are coordinates given by the mapping 
\begin{equation*}
U\overset{\mu ^{-1}}{\longrightarrow }N\overset{\left( \tilde{z},\tilde{w}%
,-p_{\tilde{x}},p_{\tilde{y}}\right) }{\longrightarrow }\mathbb{R}^{4}%
\overset{\alpha }{\longrightarrow }\mathbb{R}^{4}
\end{equation*}%
where $\alpha (a,b,c,d)=(c,d,a,b)$. To be more precise, a point $q\in U$ has
normal coordinates $(x,y,z,w)$ if and only if $q=\exp _{(0,0,z,w)}(-x,y,0,0)$%
. It follows that the lines $t\longrightarrow (at,bt,z_{0},w_{0})$ are
Hamiltonian geodesics and that $P=\left\{ y=0\right\} $.

Let us define four sets 
\begin{equation*}
S_{1}^{+}=\left\{ q\in U:\;\left\vert y\right\vert <\left\vert x\right\vert 
\text{, }x>0\right\}
\end{equation*}%
\begin{equation*}
S_{1}^{-}=\left\{ q\in U:\;\left\vert y\right\vert <\left\vert x\right\vert 
\text{, }x<0\right\}
\end{equation*}%
\begin{equation*}
S_{2}^{+}=\left\{ q\in U:\;\left\vert y\right\vert >\left\vert x\right\vert 
\text{, }y>0\right\}
\end{equation*}%
\begin{equation*}
S_{2}^{+}=\left\{ q\in U:\;\left\vert y\right\vert >\left\vert x\right\vert 
\text{, }y<0\right\} \text{.}
\end{equation*}%
Then let us put $S_{1}=S_{1}^{+}\cup S_{1}^{-}$, $S_{2}=S_{2}^{+}\cup
S_{2}^{-}$. Moreover let 
\begin{equation*}
R_{1}=\left\{ 
\begin{array}{l}
\sqrt{x^{2}-y^{2}}\;\text{on }S_{1}^{+} \\ 
-\sqrt{x^{2}-y^{2}}\;\text{on }S_{1}^{-}%
\end{array}%
\right. \text{,}
\end{equation*}%
\begin{equation*}
R_{2}=\left\{ 
\begin{array}{l}
\sqrt{y^{2}-x^{2}}\;\text{on }S_{2}^{+} \\ 
-\sqrt{y^{2}-x^{2}}\;\text{on }S_{2}^{-}%
\end{array}%
\right. \text{.}
\end{equation*}%
Now we can introduce hyperbolic cylindrical coordinates $R,\varphi ,z,w$ on $%
S_{1}$: $x=R_{1}\cosh \varphi $, $y=R_{1}\sinh \varphi $, and on $S_{2}$ and 
$x=R_{2}\sinh \varphi $, $y=R_{2}\cosh \varphi $. Clearly 
\begin{equation*}
\frac{\partial }{\partial R_{1}}=\frac{x}{R_{1}}\frac{\partial }{\partial x}+%
\frac{y}{R_{1}}\frac{\partial }{\partial y}
\end{equation*}%
is unit t.f.d. on $S_{1}^{+}$ since it is the velocity vector of the
geodesic $s\longrightarrow (s\cosh \varphi ,s\sinh \varphi ,z_{0},w_{0})$.
From similar reasons it is unit t. past directed on $S_{1}^{-}$. Also, for
instance on $S_{1}^{+}$ 
\begin{equation*}
\frac{\partial }{\partial \varphi }=y\frac{\partial }{\partial x}+x\frac{%
\partial }{\partial x}\text{,}
\end{equation*}%
however $\frac{\partial }{\partial \varphi }$ is not horizontal in general.

Below we will need the following observations. If we define $\mathcal{C}=%
\mathcal{H}^{-1}(-\frac{1}{2})$ then obviously $\Phi _{s}(\mathcal{C}%
)\subset \mathcal{C}$. Also, if $\alpha $ is the Liouville form then $\alpha 
$ restricted to $\mathcal{C}$ is preserved by the flow $\Phi _{s}$. Denote
by $(R_{0},\varphi _{0},z_{0},w_{0})$ the hyperbolic cylindrical coordinates
on $A_{\Gamma }$. Then evidently 
\begin{equation*}
\frac{\partial }{\partial \varphi }=\left( \pi \circ \Phi _{s}\right) _{\ast
}\frac{\partial }{\partial \varphi _{0}}\text{, \ \ }\frac{\partial }{%
\partial z}=\pi _{\ast }\frac{\partial }{\partial z_{0}}\text{, \ \ }\frac{%
\partial }{\partial w}=\pi _{\ast }\frac{\partial }{\partial w_{0}}\text{.}
\end{equation*}%
Moreover the fields $\frac{\partial }{\partial \varphi _{0}}$, $\frac{%
\partial }{\partial z_{0}}$, $\frac{\partial }{\partial w}$ are tangent to $%
\mathcal{C}$, and in addition $\frac{\partial }{\partial \varphi _{0}}$ is
tangent to the sets $\mathcal{C\cap }T_{(0,0,z_{0},w_{0})}^{\ast }\mathbb{R}%
^{4}$.

In what follows we will use the notion of the horizontal gradient, so now we
recall the definition. Let $f:U\longrightarrow \mathbb{R}$ be a smooth
function defined on an open subset $U$ of the sub-Lorentzian manifold $%
(M,H,g)$; \textit{the horizontal gradient }of the function $f$ is the
(horizontal) vector field $\nabla _{H}f$ determined by the condition ~$%
d_{q}f(v)=g(\nabla _{H}f(q),v)$ for every $v\in H_{q}$, $q\in U$. If $%
X_{0},...,X_{k}$ is an orthonormal frame for $(H,g)$ on $U$ with $X_{0}$
timelike, then $\nabla _{H}f=-X_{0}(f)X_{0}+X_{1}(f)X_{1}+...+X_{k}(f)X_{k}$%
. Now we can prove the following

\begin{lemma}
$\nabla _{H}R_{1}=-\frac{\partial }{\partial R_{1}}$ on $S_{1}^{+}$; in
particular $\nabla _{H}R_{1}$ is unit timelike past directed.
\end{lemma}

\begin{proof}
Fix a point $q\in $ $S_{1}^{+}$; then $q=\pi \circ \Phi
_{s}(0,0,z_{0},w_{0},-\cosh \varphi ,\sinh \varphi ,0,0)=\exp
_{(0,0,z_{0},w_{0})}s(-\cosh \varphi ,\sinh \varphi ,0,0)$ for suitable $%
\varphi $ and $s>0$. If $\lambda =(0,0,z_{0},w_{0},-\cosh \varphi ,\sinh
\varphi ,0,0)$ then $s\longrightarrow \Phi _{s}(\lambda )$ is a Hamiltonian
lift of the geodesic $\gamma (s)=\exp _{(0,0,z_{0},w_{0})}(s\lambda )$. We
have a sequence of equalities 
\begin{eqnarray*}
\left\langle \Phi _{s}(\lambda ),\tfrac{\partial }{\partial \varphi }%
\right\rangle &=&\left\langle \Phi _{s}(\lambda ),\pi _{\ast }\circ \Phi
_{s_{\ast }}\tfrac{\partial }{\partial \varphi _{0}}\right\rangle
=\left\langle \alpha _{\Phi _{s}(\lambda )},\Phi _{s_{\ast }}\tfrac{\partial 
}{\partial \varphi _{0}}\right\rangle =\left\langle \left( \Phi _{s}^{\ast
}\alpha \right) _{\lambda },\tfrac{\partial }{\partial \varphi _{0}}%
\right\rangle =\left\langle \alpha _{\lambda },\tfrac{\partial }{\partial
\varphi _{0}}\right\rangle \\
&=&\left\langle \lambda ,\pi _{\ast }\tfrac{\partial }{\partial \varphi _{0}}%
\right\rangle =0\text{.}
\end{eqnarray*}%
Also, since $\left\langle \lambda ,\frac{\partial }{\partial z}\right\rangle
=\left\langle \lambda ,\frac{\partial }{\partial w}\right\rangle =0$,
similar sequences of equalities lead to 
\begin{equation*}
\left\langle \Phi _{s}(\lambda ),\tfrac{\partial }{\partial z}\right\rangle
=\left\langle \Phi _{s}(\lambda ),\tfrac{\partial }{\partial w}\right\rangle
=0\text{.}
\end{equation*}%
Now it follows that an orthonormal basis for $(H,g)$ can be taken in the
form 
\begin{equation*}
F=\tfrac{\partial }{\partial R_{1}}\text{, \ \ }G=\tfrac{\partial }{\partial
\varphi }+r_{1}\tfrac{\partial }{\partial z}+r_{2}\tfrac{\partial }{\partial
w}\text{,}
\end{equation*}%
and the result follows from the definition of the horizontal gradient.
\end{proof}

\begin{corollary}
Geodesics $s\longrightarrow (s\cosh \varphi ,s\sinh \varphi ,z_{0},w_{0})$
are unique $U$-maximizers
\end{corollary}

\begin{proof}
As it was mentioned in the previous papers by the author, every trajectory
of a t.f.d. field of the form $\nabla _{H}f$, where $f$ is a smooth function
defined on an open set $U$ and such that $g(\nabla _{H}f,\nabla _{H}f)=const$
on $U$, is a unique $U$-maximizer.
\end{proof}

\subsection{Construction of normal forms.}

Using what we have said in the proof of lemma 2.2, there exists an
orthonormal frame $F,G$ for $(H,g)$, defined on $U\backslash \left\{
R_{i}=0\right\} =S_{1}\cup S_{2}$ with $F$ being a timelike field, which is
of the form 
\begin{equation}
F=\frac{x}{R_{1}}\frac{\partial }{\partial x}+\frac{y}{R_{1}}\frac{\partial 
}{\partial y}\text{, \ \ }G=a_{11}\frac{\partial }{\partial z}+a_{21}\frac{%
\partial }{\partial w}+\left( b_{1}+\frac{1}{R_{1}}\right) \left( y\frac{%
\partial }{\partial x}+x\frac{\partial }{\partial y}\right)  \label{Fr1}
\end{equation}
on $S_{1}$, and 
\begin{equation}
F=a_{12}\frac{\partial }{\partial z}+a_{22}\frac{\partial }{\partial w}%
+\left( b_{2}+\frac{1}{R_{2}}\right) \left( y\frac{\partial }{\partial x}+x%
\frac{\partial }{\partial y}\right) \text{, \ \ \ }G=\frac{x}{R_{2}}\frac{%
\partial }{\partial x}+\frac{y}{R_{2}}\frac{\partial }{\partial y}
\label{Fr2}
\end{equation}
on $S_{2}$, where $a_{ji}$, $b_{i}$ are smooth on $S_{i}$, $i,j=1,2$.
Indeed, first let us remark here that although all calculations in lemma 2.2
were carried out on $S_{1}^{+}$, it is not difficult to extend them to $%
S_{1}^{-}$. For instance, since $\gamma (s)=(-s\cosh \varphi ,-s\sinh
s,z_{0},w_{0})\in S_{1}^{-}$ is a timelike geodesic, the vector $\left( 
\frac{\partial }{\partial R_{1}}\right) _{\gamma (s)}=\frac{x}{R_{1}}\frac{%
\partial }{\partial x}+\frac{y}{R_{1}}\frac{\partial }{\partial y}=\dot{%
\gamma}(s)$ is also timelike on $S_{1}^{-}$. Secondly, the formula (\ref{Fr2}%
) on $S_{2}$ follows from (\ref{Fr1}) valid on $S_{1}$ by replacing the
metric $(H,g)$ with $(H,-g)$.

Below we prove the following

\begin{proposition}
There exist functions $A_{1},A_{2}\in C^{\infty }(U)$ such that\newline
$A_{1}=\left\{ 
\begin{array}{l}
\frac{a_{11}}{R_{1}}\;\text{on }S_{1} \\ 
-\frac{a_{12}}{R_{2}}\;\text{on }S_{2}%
\end{array}
\right. $ and $A_{2}=\left\{ 
\begin{array}{l}
\frac{a_{21}}{R_{1}}\;\text{on }S_{1} \\ 
-\frac{a_{22}}{R_{2}}\;\text{on }S_{2}%
\end{array}
\right. $.
\end{proposition}

The proof is similar to the proof of analogous result in \cite{g2}. On each
of the sets $S_{1}$, $S_{2}$ we will write the Hamiltonian $\mathcal{H}$,
which is the smooth function on the whole $U$. So on $S_{1}$ we have 
\begin{equation*}
2\mathcal{H}\left( x,y,z,w,p_{x},p_{y},p_{z},p_{w}\right) =-\left( \frac{x}{%
R_{1}}p_{x}+\frac{y}{R_{1}}p_{y}\right) ^{2}+\left(
a_{11}p_{z}+a_{21}p_{w}+\left( b_{1}+\frac{1}{R_{1}}\right) \left(
yp_{x}+xp_{y}\right) \right) ^{2}=
\end{equation*}%
\begin{equation*}
-p_{x}^{2}+p_{y}^{2}+\left( a_{11}p_{z}+a_{21}p_{w}\right) ^{2}+\frac{b_{1}}{%
R_{1}}\left( b_{1}R_{1}+2\right) \left( yp_{x}+xp_{y}\right) ^{2}+\frac{2}{%
R_{1}}\left( b_{1}R_{1}+1\right) \left( yp_{x}+xp_{y}\right) \left(
a_{11}p_{z}+a_{21}p_{w}\right) \,\text{,}
\end{equation*}%
while on $S_{2}$ we can write 
\begin{equation*}
2\mathcal{H}\left( x,y,z,w,p_{x},p_{y},p_{z},p_{w}\right) =-\left(
a_{12}p_{z}+a_{22}p_{w}+\left( b_{2}+\frac{1}{R_{2}}\right) \left(
yp_{x}+xp_{y}\right) \right) ^{2}+\left( \frac{x}{R_{2}}p_{x}+\frac{y}{R_{2}}%
p_{y}\right) ^{2}=
\end{equation*}%
\begin{equation*}
-p_{x}^{2}+p_{y}^{2}-\left( a_{12}p_{z}+a_{22}p_{w}\right) ^{2}-\frac{b_{2}}{%
R_{2}}\left( b_{1}R_{1}+2\right) \left( yp_{x}+xp_{y}\right) ^{2}-\frac{2}{%
R_{2}}\left( b_{2}R_{2}+1\right) \left( yp_{x}+xp_{y}\right) \left(
a_{12}p_{z}+a_{22}p_{w}\right) \text{.}
\end{equation*}

\begin{lemma}
The exist smooth function $\tilde{a}_{1},\tilde{a}_{2}:U\longrightarrow 
\mathbb{R}$ such that \newline
$\tilde{a}_{1}=\left\{ 
\begin{array}{l}
a_{11}^{2}\;\text{on }S_{1} \\ 
-a_{12}^{2}\;\text{on }S_{2}%
\end{array}
\right. $ and $\ \ \tilde{a}_{2}=\left\{ 
\begin{array}{l}
a_{21}^{2}\;\text{on }S_{1} \\ 
-a_{22}^{2}\;\text{on }S_{2}%
\end{array}
\right. $.\newline
In particular, $\tilde{a}_{1|R_{i}=0}=\tilde{a}_{2|R_{i}=0}=0$.
\end{lemma}

\begin{proof}
Indeed, using above formulas we have 
\begin{equation*}
2\mathcal{H}\left( x,y,z,w,0,0,1,0\right) =a_{11}^{2}\text{, \ \ }2\mathcal{H%
}\left( x,y,z,w,0,0,0,1\right) =a_{21}^{2}
\end{equation*}%
on $S_{1}$ and 
\begin{equation*}
2\mathcal{H}\left( x,y,z,w,0,0,1,0\right) =-a_{12}^{2}\text{, \ \ }2\mathcal{%
H}\left( x,y,z,w,0,0,0,1\right) =-a_{22}^{2}
\end{equation*}%
on $S_{2}$. Thus it is enough to define $\tilde{a}_{1}\left( x,y,z,w\right)
=2\mathcal{H}\left( x,y,z,w,0,0,1,0\right) $, $\tilde{a}_{2}\left(
x,y,z,w\right) =2\mathcal{H}\left( x,y,z,w,0,0,0,1\right) $.
\end{proof}

Next let $\left\langle \cdot ,\cdot \right\rangle $ be the Minkowski scalar
product on $U$, i.e. the one induced by the Lorentzian metric on $U$ defined
by supposing the basis $\frac{\partial }{\partial x},\frac{\partial }{%
\partial y},\frac{\partial }{\partial z},\frac{\partial }{\partial w}$ to be
orthonormal with the time orientation $\frac{\partial }{\partial x}$. For an
orthonormal basis $X,Y$ of $(H,g)$ let $\mathcal{G}=\det \left( 
\begin{array}{cc}
\left\langle X,X\right\rangle & \left\langle X,Y\right\rangle \\ 
\left\langle X,Y\right\rangle & \left\langle Y,Y\right\rangle%
\end{array}%
\right) $. Note that since the matrices from the Lorentz group have
determinant equal to $\pm 1$, $\mathcal{G}$ is independent of the choice of
the orthonormal frame $X,Y$. Clearly $\mathcal{G}$ is a smooth function on $%
U $; we will compute the values of $\mathcal{G}_{i}=\mathcal{G}_{|S_{i}}$, $%
i=1,2$. So 
\begin{equation*}
\mathcal{G}_{1}=-a_{11}^{2}-a_{21}^{2}-\left( b_{1}R_{1}+1\right) ^{2}=-%
\tilde{a}_{1|S_{1}}-\tilde{a}_{2|S_{1}}-\left( b_{1}R_{1}+1\right) ^{2}\text{%
,}
\end{equation*}%
and 
\begin{equation*}
\mathcal{G}_{2}=a_{12}^{2}+a_{22}^{2}-\left( b_{2}R_{2}+1\right) ^{2}=-%
\tilde{a}_{1|S_{2}}-\tilde{a}_{2|S_{2}}-\left( b_{2}R_{2}+1\right) ^{2}\text{%
.}
\end{equation*}%
In particular, for $R_{i}=0$ we get, by lemma 2.3, $\mathcal{G}_{i}=-1$
which means that $\mathcal{G}$ is negative on $U$. Therefore 
\begin{equation*}
b_{i}R_{i}+1=\sqrt{-\tilde{a}_{1|S_{i}}-\tilde{a}_{2|S_{i}}-\mathcal{G}_{i}}%
\text{, \ \ }i=1,2\text{.}
\end{equation*}%
Let $\tilde{d}=\left\{ 
\begin{array}{l}
b_{1}R_{1}\;\text{on }S_{1} \\ 
b_{2}R_{2}\;\text{on }S_{2} \\ 
0\text{ on }\left\{ R_{i}=0\right\}%
\end{array}%
\right. $. Then simply $\tilde{d}=\sqrt{-\tilde{a}_{1}-\tilde{a}_{2}-%
\mathcal{G}}-1$, so $\tilde{d}\in C^{\infty }(U)$ since the expression under
the root does not vanish. Now if we set $\mathcal{H}_{i}=\mathcal{H}%
_{|T_{S_{i}}^{\ast }\mathbb{R}^{4}}$, $i=1,2$, then 
\begin{equation}
\frac{\partial \mathcal{H}_{1}}{\partial p_{z}}|_{p_{w}=0}=p_{z}\tilde{a}%
_{1|S_{1}}+\frac{a_{11}}{R_{1}}\left( \tilde{d}+1\right) \left(
yp_{x}+xp_{y}\right) \text{, }  \label{r1}
\end{equation}%
and similarly 
\begin{equation}
\frac{\partial \mathcal{H}_{2}}{\partial p_{z}}|_{p_{w}=0}=p_{z}\tilde{a}%
_{1|S_{2}}-\frac{a_{12}}{R_{2}}\left( \tilde{d}+1\right) \left(
yp_{x}+xp_{y}\right) \text{.}  \label{r2}
\end{equation}%
Now let $A_{1}$ be as in the hypotheses of proposition 2.1. Then (\ref{r1})
and (\ref{r2}) become 
\begin{equation*}
\frac{\partial \mathcal{H}_{1}}{\partial p_{z}}|_{p_{w}=0}=p_{z}\tilde{a}%
_{1|S_{1}}+A_{1}\left( \tilde{d}+1\right) \left( yp_{x}+xp_{y}\right)
\end{equation*}%
and 
\begin{equation*}
\frac{\partial \mathcal{H}_{2}}{\partial p_{z}}|_{p_{w}=0}=p_{z}\tilde{a}%
_{1|S_{2}}+A_{1}\left( \tilde{d}+1\right) \left( yp_{x}+xp_{y}\right)
\end{equation*}%
and all terms, apart from $A_{1}$ perhaps, are smooth on the whole $U$.
Since $\tilde{d}+1\neq 0$, it follows that $A_{1}\left( yp_{x}+xp_{y}\right) 
$ is smooth. Setting $p_{x}=1$, $p_{y}=0$ and then $p_{x}=0$, $p_{y}=1$ we
arrive at $xA_{1},yA_{1}\in C^{\infty }(U)$. But this means $A_{1}\in
C^{\infty }(U)$ as it was stated.

Now let $A_{2}$ be defined as in the hypotheses of proposition 2.1.
Considering this time derivatives $\frac{\partial \mathcal{H}_{i}}{\partial
p_{w}}|_{p_{z}=0}$, $i=1,2$, we are led to 
\begin{equation*}
\frac{\partial \mathcal{H}_{1}}{\partial p_{w}}|_{p_{z}=0}=p_{w}\tilde{a}%
_{2}{}_{|S_{1}}+A_{2}\left( \tilde{d}+1\right) \left( yp_{x}+xp_{y}\right)
\end{equation*}
and 
\begin{equation*}
\frac{\partial \mathcal{H}_{2}}{\partial p_{w}}|_{p_{z}=0}=p_{w}\tilde{a}%
_{2|S_{2}}+A_{2}\left( \tilde{d}+1\right) \left( yp_{x}+xp_{y}\right)
\end{equation*}
which results in smoothness of $A_{2}$. The proof of proposition 2.1 is over.

\bigskip

\begin{proposition}
There exists a function $B\in C^{\infty }(U)$ such that $B=\left\{ 
\begin{array}{c}
\frac{b_{1}}{R_{1}}\;\text{on }S_{1} \\ 
-\frac{b_{2}}{R_{2}}\;\text{on }S_{2}%
\end{array}
\right. $.
\end{proposition}

\begin{proof}
Using the above formulas we can write on $S_{1}$%
\begin{equation*}
2\mathcal{H}\left( x,y,z,w,p_{x},p_{y},1,0\right) =-p_{x}^{2}+p_{y}^{2}+%
\tilde{a}_{1|S_{1}}+B\left( \tilde{d}+2\right) \left( yp_{x}+xp_{y}\right)
^{2}+2A_{1}\left( \tilde{d}+1\right) \left( yp_{x}+xp_{y}\right)
\end{equation*}%
and 
\begin{equation*}
2\mathcal{H}\left( x,y,z,w,p_{x},p_{y},1,0\right) =-p_{x}^{2}+p_{y}^{2}+%
\tilde{a}_{1|S_{2}}+B\left( \tilde{d}+2\right) \left( yp_{x}+xp_{y}\right)
^{2}+2A_{1}\left( \tilde{d}+1\right) \left( yp_{x}+xp_{y}\right)
\end{equation*}%
on $S_{2}$. Since all terms, perhaps apart from $B$, are smooth on $U$, and $%
\tilde{d}+2\neq 0$, we again arrive at $xB,yB\in C^{\infty }(U)$, which in
turn gives $B\in C^{\infty }(U)$.
\end{proof}

To conclude our considerations, similarly as in \cite{g2}, we change our
frame $F,G$ as follows: 
\begin{equation}
F\longrightarrow \frac{x}{R_{1}}F-\frac{y}{R_{1}}G\text{, \ \ }%
G\longrightarrow -\frac{y}{R_{1}}F+\frac{x}{R_{1}}G  \label{r3}
\end{equation}
on $S_{1}$ and 
\begin{equation}
F\longrightarrow \frac{y}{R_{2}}F-\frac{x}{R_{2}}G\text{, \ \ }%
G\longrightarrow -\frac{x}{R_{2}}F+\frac{y}{R_{2}}G  \label{r4}
\end{equation}
on $S_{2}$; note that both, the frame $F,G$ and our change are singular on $%
\left\{ R_{i}=0\right\} $. Carrying out calculations as indicated in (\ref%
{r3}) and (\ref{r4}), we obtain the following pre-normal form for our
structure 
\begin{equation}
\left. 
\begin{array}{c}
X=\dfrac{\partial }{\partial x}-yB\left( y\dfrac{\partial }{\partial x}+x%
\dfrac{\partial }{\partial y}\right) -yA_{1}\dfrac{\partial }{\partial z}%
-yA_{2}\dfrac{\partial }{\partial w} \\ 
Y=\dfrac{\partial }{\partial y}+xB\left( y\dfrac{\partial }{\partial x}+x%
\dfrac{\partial }{\partial y}\right) +xA_{1}\dfrac{\partial }{\partial z}%
+xA_{2}\dfrac{\partial }{\partial w}%
\end{array}
\right.  \label{preNorm}
\end{equation}
where $X$ is a time orientation and $A_{1},A_{2},B$ are smooth around the
origin. The only thing that may require some explanations is the fact that $%
X $ is a time orientation. However it is clear that for a timelike field $X$
to be a time orientation it suffices to be future directed at a single
point, and surely $X=\frac{x}{R_{1}}F-\frac{y}{R_{1}}G$ is future directed
on $S_{1}^{+}$, since $F$ is.

In order to be able to find some additional conditions that can be imposed
on $A_{1}$, $A_{2}$ in (\ref{preNorm}) we have to use our assumptions. First
let us note that, by construction, $\frac{\partial }{\partial w}_{|\Gamma }=%
\frac{\partial }{\partial \tilde{w}}_{|\Gamma }$ is transverse to $H^{2}$,
so $\frac{\partial }{\partial w}$ is transverse to $H^{2}$ on $U$ (recall
that $U$ is sufficiently small).

We compute the commutator of $X$ and $Y$ to be equal to 
\begin{equation*}
\left[ X,Y\right] =I\frac{\partial }{\partial x}+II\frac{\partial }{\partial
y}+III\frac{\partial }{\partial z}+IV\frac{\partial }{\partial w}
\end{equation*}%
where\newline
$I=y\left( 3B+x\frac{\partial B}{\partial x}+y\frac{\partial B}{\partial y}%
+\left( x^{2}-y^{2}\right) B^{2}\right) $, \newline
$II=x\left( 3B+x\frac{\partial B}{\partial x}+y\frac{\partial B}{\partial y}%
+\left( x^{2}-y^{2}\right) B^{2}\right) $, $\newline
III=2A_{1}+x\frac{\partial A_{1}}{\partial x}+y\frac{\partial A_{1}}{%
\partial y}+(x^{2}-y^{2})A_{1}B$, and \newline
$IV=2A_{2}+x\frac{\partial A_{2}}{\partial x}+y\frac{\partial A_{2}}{%
\partial y}+(x^{2}-y^{2})A_{2}B$. \newline
Now, since $X_{|\Gamma }=\frac{\partial }{\partial x}$, $Y_{|\Gamma }=\frac{%
\partial }{\partial y}$, $\dim H_{(0,0,z,w)}^{2}=3$, and as it was noticed $%
\frac{\partial }{\partial w}$ is transverse to $H^{2}$, it follows that $%
III_{|\Gamma }=2A_{1|\Gamma }$ does not vanish. We renormalize the $z$-axis
by making the following change of coordinates: $(x,y,z,w)\longrightarrow
(x,y,\alpha (z,w),w)$, where $\alpha $ solves the equation ($w$ is a
parameter here) 
\begin{equation*}
\alpha (z,w)+z\frac{d}{dz}\alpha (z,w)=\frac{1}{2A_{1}(0,0,z,w)}\text{.}
\end{equation*}%
In this way we keep the form (\ref{preNorm}) and, in the new coordinates, $%
A_{1}(0,0,z,w)=-\frac{1}{2}$. Now setting $\psi _{1}=-2A_{1}-1$ we obtain
proposition 1.1.

Before we proceed we prove the following lemma.

\begin{lemma}
$\frac{\partial }{\partial z}_{|P}$ is tangent to $H^{2}$.
\end{lemma}

\begin{proof}
If we look closer at coordinates $\tilde{x},\tilde{y},\tilde{z},\tilde{w}$
and $x,y,z,w$, it is seen that $\frac{\partial }{\partial z}=\frac{\partial 
}{\partial \tilde{z}}$. Indeed, $\frac{\partial }{\partial z}=\frac{\partial 
\tilde{x}}{\partial z}\frac{\partial }{\partial \tilde{x}}+\frac{\partial 
\tilde{y}}{\partial z}\frac{\partial }{\partial \tilde{y}}+\frac{\partial 
\tilde{z}}{\partial z}\frac{\partial }{\partial \tilde{z}}+\frac{\partial 
\tilde{w}}{\partial z}\frac{\partial }{\partial \tilde{w}}$, and $z=\tilde{z}
$, $w=\tilde{w}$. Therefore it is enough to carry out all computations in
the first set of coordinates. To this end fix a point $q=(\tilde{x},0,\tilde{%
z},\tilde{w})=g_{\tilde{X}}^{\tilde{x}}\circ g_{[\tilde{X},\tilde{Y}]}^{%
\tilde{z}}\sigma (\tilde{w})$ belonging to $P$. Then 
\begin{equation*}
\left( \tfrac{\partial }{\partial \tilde{z}}\right) _{q}=\tfrac{d}{ds}%
|_{s=0}g_{\tilde{X}}^{\tilde{x}}\circ g_{[\tilde{X},\tilde{Y}]}^{\tilde{z}%
+s}\sigma (\tilde{w})=\left( dg_{\tilde{X}}^{\tilde{x}}\right) [\tilde{X},%
\tilde{Y}]_{\left( g_{[\tilde{X},\tilde{Y}]}^{\tilde{z}}\sigma (\tilde{w}%
)\right) }
\end{equation*}%
Now let $\gamma (t)=g_{\tilde{X}}^{t}\circ g_{[\tilde{X},\tilde{Y}]}^{\tilde{%
z}}\sigma (\tilde{w})$, i.e. $\gamma $ is the abnormal curve passing through 
$g_{\left[ X,Y\right] }^{\tilde{z}}\sigma (\tilde{w})$ at time $t=0$. Let
moreover $\lambda (t)$ be an abnormal lift of $\gamma $ satisfying PMP,
Pontriagin maximum principle, that is to say $(\gamma (t),\lambda (t))$ is
an abnormal biextremal. Then clearly $\lambda (t)\in (H_{\gamma
(t)}^{2})^{\perp }\subset T^{\ast }\mathbb{R}^{4}$ where the latter stands
for the annihilator of $H_{\gamma (t)}^{2}$ - cf. \cite{Suss}. This in
particular implies that $H_{\gamma (t)}^{2}=\ker \lambda (t)$. Further, from
the proof of PMP - see \cite{Agr} - it follows that $\lambda (t)=\left( dg_{%
\tilde{X}}^{-t}\right) ^{\ast }\lambda (0)$ for every $t$. Thus, taking all
above-mentioned facts together we obtain 
\begin{eqnarray}
\left\langle \lambda (t),\left( \tfrac{\partial }{\partial \tilde{z}}\right)
_{\gamma (t)}\right\rangle &=&\left\langle \lambda (t),\left( dg_{\tilde{X}%
}^{t}\right) [\tilde{X},\tilde{Y}]\left( g_{[\tilde{X},\tilde{Y}]}^{\tilde{z}%
}\sigma (\tilde{w})\right) \right\rangle =  \label{r5} \\
\left\langle \left( dg_{\tilde{X}}^{t}\right) ^{\ast }\lambda (t),[\tilde{X},%
\tilde{Y}]_{\left( g_{[\tilde{X},\tilde{Y}]}^{\tilde{z}}\sigma (\tilde{w}%
)\right) }\right\rangle &=&\left\langle \lambda (0),[\tilde{X},\tilde{Y}%
]_{\left( g_{[\tilde{X},\tilde{Y}]}^{\tilde{z}}\sigma (\tilde{w})\right)
}\right\rangle =0  \notag
\end{eqnarray}%
which terminates the proof.
\end{proof}

Now let us see what happens on $P$. So \newline
$\left[ X,Y\right] _{|P}=x\left( 3B+x\frac{\partial B}{\partial x}%
+x^{2}B\right) \frac{\partial }{\partial y}+\left( 2A_{1}+x\frac{\partial
A_{1}}{\partial x}+x^{2}A_{1}B\right) \frac{\partial }{\partial z}+\left(
2A_{2}+x\frac{\partial A_{2}}{\partial x}+x^{2}A_{2}B\right) \frac{\partial 
}{\partial w}$, \newline
$X_{|P}=\frac{\partial }{\partial x}$, and\newline
$Y_{|P}=\left( 1+x^{2}B\right) \tfrac{\partial }{\partial y}+xA_{1}\tfrac{%
\partial }{\partial z}+xA_{2}\tfrac{\partial }{\partial w}$,\newline
where $A_{1}$, $A_{2}$, $B$ are evaluated at $(x,0,z,w)$. Calculations give 
\begin{equation*}
\left[ X,Y\right] _{|P}-\frac{x\left( 3B+x\frac{\partial B}{\partial x}%
+x^{2}B\right) }{\left( 1+x^{2}B\right) }Y_{|P}=\frac{B^{2}x^{4}-Bx^{4}+2}{%
Bx^{2}+1}\left( A_{1}\frac{\partial }{\partial z}+A_{2}\frac{\partial }{%
\partial w}\right)
\end{equation*}%
from which it follows that $A_{1}\frac{\partial }{\partial z}+A_{2}\frac{%
\partial }{\partial w}$ is tangent to $H^{2}$ on $P$. But since $\frac{%
\partial }{\partial z}$ is also tangent to $H^{2}$ it follows that $A_{2}%
\frac{\partial }{\partial w}$ is tangent to $H^{2}$ at point of $P$ which is
possible only when $A_{2}(x,0,z,w)=0$ identically. This means that $A_{2}$
may be replaced by $yA_{2}$ for some other smooth function $A_{2}$. Thus we
are led to 
\begin{equation*}
\left. 
\begin{array}{c}
X=\dfrac{\partial }{\partial x}-yB\left( y\dfrac{\partial }{\partial x}+x%
\dfrac{\partial }{\partial y}\right) -yA_{1}\dfrac{\partial }{\partial z}%
-y^{2}A_{2}\dfrac{\partial }{\partial w} \\ 
Y=\dfrac{\partial }{\partial y}+xB\left( y\dfrac{\partial }{\partial x}+x%
\dfrac{\partial }{\partial y}\right) +xA_{1}\dfrac{\partial }{\partial z}%
+xyA_{2}\dfrac{\partial }{\partial w}%
\end{array}%
\right.
\end{equation*}%
with $A_{1}(0,0,z,w)=\frac{1}{2}$.

Let $adX.Y=\left[ X,Y\right] $, and $ad^{k+1}X.Y=\left[ X,ad^{k}X.Y\right] $%
, $k=1,2,...$ Now we will extract some more information about the
commutators of the fields $X$, $Y$.

\begin{lemma}
$ad^{n}X.Y_{|P}$ is tangent to $H_{|P}^{2}$, $n=0,1,2,...$
\end{lemma}

\begin{proof}
We use similar considerations as in the proof of lemma 2.4. Consider the
abnormal curve $\gamma (t)$ starting from a point $\gamma
(0)=q=(x_{0},0,z_{0},w_{0})$. Let $\lambda (t)$ be the abnormal lift of $%
\gamma $ satisfying PMP; then clearly $H_{\gamma (t)}^{2}=\ker \lambda (t)$
and, again by \cite{Agr}, $\lambda (t)=\left( dg_{X}^{-t}\right) ^{\ast
}\lambda (0)$ for every $t$. Now for each $t$, $\left\vert t\right\vert $
sufficiently small, and every integer $n$ we have 
\begin{equation*}
0=\left\langle \lambda (t),Y_{\gamma (t)}\right\rangle =\left\langle \lambda
(0),\left( dg_{X}^{-t}\right) Y_{q}\right\rangle =\left\langle \lambda
(0),\tsum\nolimits_{k=0}^{n}\frac{t^{k}}{k!}\left( ad^{k}X.Y\right)
_{q}\right\rangle +o(t^{n})\text{,}
\end{equation*}%
and because $\left\langle \lambda (0),Y_{q}\right\rangle =0$, the result
follows by induction.
\end{proof}

As a consequence of lemma 2.4 we know that $X,Y,\left[ X,Y\right] ,\left[ Y,%
\left[ X,Y\right] \right] $ are linearly independent everywhere. We will
examine the $w$-coordinate of $\left[ Y,\left[ X,Y\right] \right] $ on $%
\Gamma $. The most convenient way is to treat $\left[ Y,\left[ X,Y\right] %
\right] $ as an operator: $\left[ Y,\left[ X,Y\right] \right] =-Y^{2}\circ
X+2Y\circ X\circ Y-X\circ Y^{2}$. Only result on $\Gamma $ interests us, so
it is enough to carry out computations as follows:\newline
$Y^{2}\circ X(w)=Y^{2}(-y^{2}A_{2})=-2(1+x^{2}B)^{2}A_{2}+O(y)$,\newline
$Y\circ X\circ Y(w)=Y\circ X(xyA_{2})=(1+x^{2}B)(1-y^{2}B)A_{2}(A_{2}+x\frac{%
\partial A_{2}}{\partial x})+O(y)+O(x)$,\newline
$X\circ Y^{2}(w)=X\circ Y(xyA_{2})=(1+x^{2}B)(1-y^{2}B)(A_{2}+y\frac{%
\partial A_{2}}{\partial y})+O(y)+O(x)$.\newline
Consequently 
\begin{equation}
\left[ Y,\left[ X,Y\right] \right] (w)_{|\Gamma }=2A_{2}(0,0,z,w)\text{.}
\label{r6}
\end{equation}%
(\ref{r6}) does not vanish, so we renormalize the $w$-axis by making the
change $(x,y,z,w)\longrightarrow (x,y,z,\beta (w)w)$, where $\beta $ is a
solves the equation 
\begin{equation*}
\beta (w)+w\frac{d}{dw}\beta (w)=\frac{1}{2A_{2}(0,0,0,w)}\text{.}
\end{equation*}%
In the new coordinates $A_{2}(0,0,0,w)=-\frac{1}{2}$, and to end the proof
we set $\varphi =-B$, $\psi _{i}=-2A_{i}-1$, $i=1,2$.

\subsection{Remarks.}

Having proved theorem 1.1 it is seen why the structure from example 1.1 is
called flat: every structure in the normal form (\ref{NF}) can be regarded
as a perturbation of the flat structure. Moreover (cf. \cite{Bel}) we see
that the flat Engel structure is the nilpotent approximation for general
Engel structures of Hamiltonian type given by (\ref{NF}). In particular, if
we assign weights to coordinates in the following way $weight(x)=weight(y)=1$%
, $weight(z)=2$, $weight(w)=3$, then the fields defining the flat structure
are homogeneous of degree $-1$.

\section{Reachable sets.}

In his previous papers the author managed to find a sort of algorithm
allowing to compute functions describing reachable sets - see \cite{g7}, 
\cite{g8}, \cite{g9}. The case considered in the present paper, however, is
more complicated and the mentioned methods do not work. Therefore one must
content oneself only with certain estimates on the reachable sets.

In section 2 we recalled the definition of the horizontal gradient of a
smooth function. Notice that if $\gamma :[a,b]\longrightarrow U$ is a
nspc.f.d. curve and a smooth function $f$ is such that $\nabla _{H}f$ is
null f.d. on $U$, then $t\longrightarrow f(\gamma (t))$ is nonincreasing.

\subsection{Geometric optimality of abnormal curves.}

Let $(H,g)$ be an Engel sub-Lorentzian structure of Hamiltonian type which
generated on an open set $U\subset \mathbb{R}^{4}$ by the frame in the
normal form 
\begin{equation}
\left. 
\begin{array}{c}
X=\dfrac{\partial }{\partial x}+y\varphi \left( y\dfrac{\partial }{\partial x%
}+x\dfrac{\partial }{\partial y}\right) +\frac{1}{2}y(1+\psi _{1})\dfrac{%
\partial }{\partial z}+\frac{1}{2}y^{2}(1+\psi _{2})\dfrac{\partial }{%
\partial w} \\ 
Y=\dfrac{\partial }{\partial y}-x\varphi \left( y\dfrac{\partial }{\partial x%
}+x\dfrac{\partial }{\partial y}\right) -\frac{1}{2}x(1+\psi _{1})\dfrac{%
\partial }{\partial z}-\frac{1}{2}xy(1+\psi _{2})\dfrac{\partial }{\partial w%
}%
\end{array}%
\right. \text{,}  \label{EngSpec}
\end{equation}%
where we additionally suppose that $\varphi =\varphi (x,y,w)$ and $\psi
_{2}=\psi _{2}(x,y,w)$ i.e. $\varphi $ and $\psi _{2}$ do not depend on $z$.
Consider now a projection $p:\mathbb{R}^{4}\longrightarrow \mathbb{R}^{3}$, $%
p(x,y,z,w)=(x,y,w)$. (\ref{EngSpec}) is mapped by $p$ to the frame 
\begin{equation}
\left. 
\begin{array}{c}
\tilde{X}=\dfrac{\partial }{\partial x}+y\varphi \left( y\dfrac{\partial }{%
\partial x}+x\dfrac{\partial }{\partial y}\right) +\frac{1}{2}y^{2}(1+\psi
_{2})\dfrac{\partial }{\partial w} \\ 
\tilde{Y}=\dfrac{\partial }{\partial y}-x\varphi \left( y\dfrac{\partial }{%
\partial x}+x\dfrac{\partial }{\partial y}\right) -\frac{1}{2}xy(1+\psi _{2})%
\dfrac{\partial }{\partial w}%
\end{array}%
\right.  \label{FrPr}
\end{equation}%
on the open set $\tilde{U}=p(U)\subset \mathbb{R}^{3}$. If $\tilde{H}=Span\{%
\tilde{X},\tilde{Y}\}$ and $\tilde{g}$ is a metric on $\tilde{H}$ defined by
assuming $\tilde{X},\tilde{Y}$ to be an orthonormal frame with a time
orientation $\tilde{X}$, then we obtain the mapping 
\begin{equation*}
p:\left( U,H,g\right) \longrightarrow (\tilde{U},\tilde{H},\tilde{g})
\end{equation*}%
of sub-Lorentzian manifolds with the property that $d_{q}p_{|H_{q}}:H_{q}%
\longrightarrow \tilde{H}_{p(q)}$ is an isometry for every $q\in U$.
Obviously, the image under $p$ of a nspc.f.d. (t.f.d., null f.d.) curve with
respect to $\left( H,g\right) $ is a nspc.f.d. (t.f.d., null f.d.) curve
with respect to $(\tilde{H},\tilde{g})$. Conversely, if $\tilde{\gamma}%
(t)=(x(t),y(t),w(t))$ is a nspc.f.d. (t.f.d., null f.d.) curve on $(\tilde{U}%
,\tilde{H},\tilde{g})$ and $q_{0}\in p^{-1}(\tilde{\gamma}(0))\cap U$, then
there exists exactly one nspc.f.d. (t.f.d., null f.d.) curve $\gamma (t)$ on 
$\left( U,H,g\right) $ such that $p(\gamma (t))=\tilde{\gamma}(t)$, $\gamma
(0)=q_{0}$. Indeed, the $z$-coordinate of $\gamma $ is computed from $\dot{z}%
=\frac{1}{2}(\dot{y}x-x\dot{y})$. One of the immediate consequences of this
reasoning is the relation 
\begin{equation*}
p\left( J^{+}(q_{0},U)\right) =\tilde{J}^{+}(p(q_{0}),\tilde{U})\text{,}
\end{equation*}%
where by $\tilde{J}^{+}(p(q_{0}),\tilde{U})$ we denote the corresponding
reachable set for the structure $(\tilde{H},\tilde{g})$. The other is
enclosed in the proposition below. Recall that $\partial _{U}$ (resp. $%
\partial _{\tilde{U}}$) denotes the boundary with respect to $U$ (resp. to $%
\tilde{U}$).

\begin{proposition}
Suppose that $\tilde{\gamma}:[0,T]\longrightarrow \tilde{U}$, $\tilde{\gamma}%
(0)=\tilde{q}_{0}$, is geometrically optimal, i.e. $\tilde{\gamma}\left(
[0,T]\right) \subset \partial _{\tilde{U}}\tilde{J}^{+}(\tilde{q}_{0},\tilde{%
U})$. Chose $q_{0}\in p^{-1}(\tilde{q}_{0})$, and let $\gamma
:[0,T]\longrightarrow U$, $\gamma (0)=q_{0}$, be the lift described above.
Then $\gamma $ is also geometrically optimal, i.e. $\gamma \left( \lbrack
0,T]\right) \subset \partial _{U}J^{+}(q_{0},U)$.
\end{proposition}

\begin{proof}
Suppose that $\tilde{\gamma}:[0,T]\longrightarrow \tilde{U}$ is
geometrically optimal but $\gamma \left( \lbrack 0,T]\right) \subset int$ $%
J^{+}(q_{0},U)$. Take an open set $V$ in $U$ such that $\gamma (T)\in V$ and 
$V\subset J^{+}(q_{0},U)$. Then $\tilde{\gamma}(T)\in p(V)$ where the latter
set is open and contained in $\tilde{J}^{+}(\tilde{q}_{0},\tilde{U})$. This
contradicts the geometric optimality of $\tilde{\gamma}$ and the proof is
over.
\end{proof}

\begin{corollary}
The abnormal t.f.d. curve starting from the origin is geometrically optimal
for $(U,H,g)$. Consequently, the set $I^{+}(0,U)$ is not open, and $%
N^{+}(0,U)$ is not closed.
\end{corollary}

\begin{proof}
It is enough to notice that the frame (\ref{FrPr}) defining $(\tilde{U},%
\tilde{H},\tilde{g})$ is given in the normal form for Martinet
sub-Lorentzian structures of Hamiltonian type considered in \cite{g8}. Thus,
using \cite{g8}, we know that the abnormal curve for $(\tilde{H},\tilde{g})$
initiating at the origin is geometrically optimal. Now proposition above
applies. The second part is clear - cf. \cite{g6}.
\end{proof}

In particular this proves proposition 1.2. As it was mentioned above,
abnormal timelike curves always satisfy necessary conditions for optimality,
so the presented method may prove to be useful in applications.

\subsection{Some estimates in the flat case.}

Recall that the flat Engel sub-Lorentzian structure is, by definition, the
structure defined by an orthonormal frame $X=\frac{\partial }{\partial x}+%
\frac{1}{2}y\frac{\partial }{\partial z}+\frac{1}{2}y^{2}\frac{\partial }{%
\partial w}$, $Y=\frac{\partial }{\partial y}-\frac{1}{2}x\frac{\partial }{%
\partial z}-\frac{1}{2}xy\frac{\partial }{\partial w}$ where $X$ is a time
orientation. We see that $X(x,y,z,0)$, $Y(x,y,z,0)$ determine, in the space $%
\mathbb{R}^{3}(x,y,z)$, the Heisenberg sub-Lorentzian metric considered in 
\cite{g7}, while $X(x,y,0,w)$, $Y(x,y,0,w)$ stipulate, in the space $\mathbb{%
R}^{3}(x,y,w)$, the flat Martinet sub-Lorentzian structure investigated in 
\cite{g8}. This leads us to considering the following Cauchy problems.
Similarly as in the mentioned papers let $\Gamma _{1}$ be the hyperplane $%
\left\{ y=x\right\} $, and $\Gamma _{2}$ be the hyperplane $\left\{
y=-x\right\} $. Consider the following Cauchy problems (cf. \cite{g7}): 
\begin{equation}
(X-Y)(\eta )=0\text{, \ \ }\eta _{|\Gamma _{1}}=z\text{,}  \label{Ca1}
\end{equation}%
\begin{equation}
(X-Y)(\eta )=0\text{, \ \ }\eta _{|\Gamma _{2}}=-z\text{.}  \label{Ca2}
\end{equation}%
Their solutions are respectively \newline
$\hat{f}_{1}(x,y,z,w)=z-\frac{1}{4}(x^{2}-y^{2})$, \newline
$\hat{f}_{2}(x,y,z,w)=-z-\frac{1}{4}(x^{2}-y^{2})$. \newline
The horizontal gradients are computed to be \newline
$\nabla _{H}\hat{f}_{1}=\frac{1}{2}(x-y)(X-Y)$, \newline
$\nabla _{H}\hat{f}_{2}=\frac{1}{2}(x+y)(X+Y)$. \newline
Note that $\nabla _{H}\hat{f}_{i}$ is null f.d. on $\left\{ \left\vert
y\right\vert <x\right\} $, $i=1,2$.

Next consider the following Cauchy problems (cf. \cite{g8}): 
\begin{equation}
(X-Y)(\eta )=0\text{, \ \ }\eta _{|\Gamma _{1}}=w\text{,}  \label{Ca3}
\end{equation}%
\begin{equation}
(X+Y)(\eta )=0\text{, \ \ }\eta _{|\Gamma _{2}}=w\text{,}  \label{Ca4}
\end{equation}%
\begin{equation}
(X+Y)(\eta )=0\text{, \ \ }\eta _{|y=0}=-w\text{,}  \label{Ca5}
\end{equation}%
\begin{equation}
(X-Y)(\eta )=0\text{, \ \ }\eta _{|y=0}=-w\text{.}  \label{Ca6}
\end{equation}%
Their solutions are: \newline
$\hat{g}_{1}(x,y,z,w)=w-\frac{1}{16}(x^{2}-y^{2})(x+3y)$, \newline
$\hat{g}_{2}(x,y,z,w)=w-\frac{1}{16}(x^{2}-y^{2})(x-3y)$, \newline
$\hat{g}_{3}(x,y,z,w)=-w-\frac{1}{4}(xy^{2}-y^{3})$, \newline
$\hat{g}_{4}(x,y,z,w)=-w-\frac{1}{4}(xy^{2}+y^{3})$, \newline
respectively. Their horizontal gradients are:\newline
$\nabla _{H}\hat{g}_{1}=\frac{3}{16}(x-y)(x+3y)(X-Y)$, \newline
$\nabla _{H}\hat{g}_{2}=\frac{3}{16}(x+y)(x-3y)(X+Y)$, \newline
$\nabla _{H}\hat{g}_{3}=\frac{3}{4}y^{2}(X+Y)$, \newline
$\nabla _{H}\hat{g}_{4}=\frac{3}{4}y^{2}(X-Y)$.\newline
It is easy to see \cite{g8} that all $\nabla _{H}\hat{g}_{i}$ are null
fields, and $\nabla _{H}\hat{g}_{1}$ is f.d. on the set $\left\{ -\frac{1}{3}%
x<y<x\text{, }x>0\right\} $, $\nabla _{H}\hat{g}_{2}$ is f.d. on $\left\{
-x<y<\frac{1}{3}x\text{, }x>0\right\} $, finally $\nabla _{H}\hat{g}_{3}$
and $\nabla _{H}\hat{g}_{4}$ are f.d. on $\left\{ y\neq 0\text{, }%
x>0\right\} $. Let us define the following subsets of $\mathbb{R}^{4}$: 
\begin{equation*}
A_{11}=\left\{ \hat{f}_{1}\leq 0\right\} \cap \left\{ \hat{g}_{1}\leq
0\right\} \cap \left\{ x\geq 0\text{, }y\geq 0\text{, }z\geq 0\text{, }w\geq
0\right\} \text{,}
\end{equation*}%
\begin{equation*}
A_{12}=\left\{ \hat{f}_{1}\leq 0\right\} \cap \left\{ \hat{g}_{2}\leq
0\right\} \cap \left\{ x\geq 0\text{, }y\leq 0\text{, }z\geq 0\text{, }w\geq
0\right\} \text{,}
\end{equation*}%
\begin{equation*}
A_{13}=\left\{ \hat{f}_{1}\leq 0\right\} \cap \left\{ \hat{g}_{3}\leq
0\right\} \cap \left\{ x\geq 0\text{, }y\geq 0\text{, }z\geq 0\text{, }w\leq
0\right\} \text{,}
\end{equation*}%
\begin{equation*}
A_{14}=\left\{ \hat{f}_{1}\leq 0\right\} \cap \left\{ \hat{g}_{4}\leq
0\right\} \cap \left\{ x\geq 0\text{, }y\leq 0\text{, }z\geq 0\text{, }w\leq
0\right\} \text{,}
\end{equation*}%
\begin{equation*}
A_{21}=\left\{ \hat{f}_{2}\leq 0\right\} \cap \left\{ \hat{g}_{1}\leq
0\right\} \cap \left\{ x\geq 0\text{, }y\geq 0\text{, }z\leq 0\text{, }w\geq
0\right\} \text{,}
\end{equation*}%
\begin{equation*}
A_{22}=\left\{ \hat{f}_{2}\leq 0\right\} \cap \left\{ \hat{g}_{2}\leq
0\right\} \cap \left\{ x\geq 0\text{, }y\leq 0\text{, }z\leq 0\text{, }w\geq
0\right\} \text{,}
\end{equation*}%
\begin{equation*}
A_{23}=\left\{ \hat{f}_{2}\leq 0\right\} \cap \left\{ \hat{g}_{3}\leq
0\right\} \cap \left\{ x\geq 0\text{, }y\geq 0\text{, }z\leq 0\text{, }w\leq
0\right\} \text{,}
\end{equation*}%
\begin{equation*}
A_{24}=\left\{ \hat{f}_{2}\leq 0\right\} \cap \left\{ \hat{g}_{4}\leq
0\right\} \cap \left\{ x\geq 0\text{, }y\leq 0\text{, }z\leq 0\text{, }w\leq
0\right\} \text{.}
\end{equation*}%
Using corollary 3.1, the remark on horizontal gradients from the beginning
of this section, and computations made in \cite{g7}, \cite{g8} we obtain

\begin{proposition}
Let $J^{+}(0)$ be the reachable set from zero for the flat Engel structure.
Then 
\begin{equation*}
J^{+}(0)\subset \tbigcup\nolimits_{i=1,2}\tbigcup\nolimits_{j=1,...,4}A_{ij}%
\text{.}
\end{equation*}
\end{proposition}

\begin{remark}
It should be mentioned that our flat case was treated by Krener and
Schattler in \cite{Krener}. If $Z_{1},Z_{2},Z_{3}$ are vector fields, denote
by $Z_{1}Z_{2}Z_{3}$ the curve starting from the origin which is a
concatenation of a segment of the trajectory of $Z_{1}$ starting from the
origin with a segment of a trajectory of $Z_{2}$ and a segment of a
trajectory of $Z_{3}$. The authors observed that geometrically optimal
curves are the following concatenations: $\left( X+Y\right) X\left(
X+Y\right) $, $\left( X+Y\right) X\left( X-Y\right) $, $\left( X-Y\right)
X\left( X-Y\right) $, $\allowbreak \left( X-Y\right) X\left( X+Y\right) $,
and also $\left( X+Y\right) \left( X-Y\right) \left( X+Y\right) $, $\left(
X-Y\right) \left( X+Y\right) \left( X-Y\right) $ where in the last series
the following restriction on time applies: the time along the intermediate
arc is greater than or equal to the sum of times along the first and the
last arc.
\end{remark}

\subsection{Some estimates in the general case.}

Now consider a structure generated by the frame $X,Y$ as in (\ref{EngSpec}),
i.e. $\varphi =\varphi (x,y,w)$ and $\psi _{2}=\psi _{2}(x,y,w)$ but
additionally assume that all objects are real analytic. Fix a normal
neighbourhood $U$ of the origin and consider in $U$ the Cauchy problems (\ref%
{Ca1}),..., (\ref{Ca6}) (where $X,Y$ are as in (\ref{EngSpec})). Denote
respective solutions by $f_{i}$, $i=1,2$, and $g_{j}$, $j=1,...,4$. Again
according to \cite{g7}, \cite{g8} $f_{i}=\hat{f}_{i}+O(r^{3})$, $i=1,2$, and 
$g_{j}=\hat{g}_{j}+O(r^{4})$, $j=1,...,4$, where $r=\sqrt{%
x^{2}+y^{2}+z^{2}+w^{2}}$. Also the horizontal gradients keep the suitable
signs, provided $U$ is sufficiently small. Now take a semi-analytic set $%
\Sigma $ from theorem 1.2 in \cite{g8}. Considering $\Sigma $ as a subset of 
$\mathbb{R}^{4}$, $\Sigma $ becomes a set of dimension $3$, and hence $U\cap
\left\{ x\geq 0\right\} \backslash \Sigma $ has two connected components $%
\Sigma ^{+}$ and $\Sigma ^{-}$. Let us agree that $\Sigma ^{+}$ contains the
trajectory of $X+Y$ starting from $0$. Now, if we define 
\begin{equation*}
A_{11}=\left\{ f_{1}\leq 0\right\} \cap \left\{ g_{1}\leq 0\right\} \cap
\Sigma ^{+}\cap U\cap \left\{ x\geq 0\text{, }z\geq 0\text{, }w\geq
0\right\} \text{,}
\end{equation*}%
\begin{equation*}
A_{12}=\left\{ f_{1}\leq 0\right\} \cap \left\{ g_{2}\leq 0\right\} \cap
\Sigma ^{-}\cap U\cap \left\{ x\geq 0\text{, }z\geq 0\text{, }w\geq
0\right\} \text{,}
\end{equation*}%
\begin{equation*}
A_{13}=\left\{ f_{1}\leq 0\right\} \cap \left\{ g_{3}\leq 0\right\} \cap
U\cap \left\{ x\geq 0\text{, }y\geq 0\text{, }z\geq 0\text{, }w\leq
0\right\} \text{,}
\end{equation*}%
\begin{equation*}
A_{14}=\left\{ f_{1}\leq 0\right\} \cap \left\{ g_{4}\leq 0\right\} \cap
U\cap \left\{ x\geq 0\text{, }y\leq 0\text{, }z\geq 0\text{, }w\leq
0\right\} \text{,}
\end{equation*}%
\begin{equation*}
A_{21}=\left\{ f_{2}\leq 0\right\} \cap \left\{ g_{1}\leq 0\right\} \cap
\Sigma ^{+}\cap U\cap \left\{ x\geq 0\text{, }z\leq 0\text{, }w\geq
0\right\} \text{,}
\end{equation*}%
\begin{equation*}
A_{22}=\left\{ f_{2}\leq 0\right\} \cap \left\{ g_{2}\leq 0\right\} \cap
\Sigma ^{-}\cap U\cap \left\{ x\geq 0\text{, }z\leq 0\text{, }w\geq
0\right\} \text{,}
\end{equation*}%
\begin{equation*}
A_{23}=\left\{ f_{2}\leq 0\right\} \cap \left\{ g_{3}\leq 0\right\} \cap
U\cap \left\{ x\geq 0\text{, }y\geq 0\text{, }z\leq 0\text{, }w\leq
0\right\} \text{,}
\end{equation*}%
\begin{equation*}
A_{24}=\left\{ f_{2}\leq 0\right\} \cap \left\{ g_{4}\leq 0\right\} \cap
U\cap \left\{ x\geq 0\text{, }y\leq 0\text{, }z\leq 0\text{, }w\leq
0\right\} \text{,}
\end{equation*}%
then again using corollary 3.1 and computations from \cite{g7}, \cite{g8} we
get

\begin{proposition}
Let $J^{+}(0,U)$ be the reachable set from zero for analytic Engel structure
as in (\ref{EngSpec}). Then 
\begin{equation*}
J^{+}(0,U)\subset
\tbigcup\nolimits_{i=1,2}\tbigcup\nolimits_{j=1,...,4}A_{ij}\text{.}
\end{equation*}
\end{proposition}

Note that if we do not assume $\varphi =\varphi (x,y,w)$, $\psi _{2}=\psi
_{2}(x,y,w)$ (i.e. we are not sure about the geometric optimality of the
abnormal curve initiating at zero) then we know that 
\begin{equation*}
J^{+}(0,U)\subset \left( \left\{ f_{1}\leq 0\text{, }x\geq 0\text{, }z\geq
0\right\} \cup \left\{ f_{2}\leq 0\text{, }x\geq 0\text{, }z\leq 0\right\}
\right) \cap U\text{.}
\end{equation*}

\subsection{Final remarks.}

First let us notice that for all Engel sub-Lorentzian structures the
half-lines $\left\{ y=\pm x\text{, }z=w=0\right\} $ (in coordinates from
theorem 1.1) are geometrically optimal.\ Indeed, this follows from theorem
proved in \cite{g6} and asserting that null f.d. Hamiltonian geodesics are
geometrically optimal.

Secondly, in all cases treated in proposition 3.1, the curves $\left(
X-Y\right) \left( X+Y\right) $, $\left( X+Y\right) \left( X-Y\right) $, $%
X\left( X-Y\right) $, $X\left( X+Y\right) $ (we use here the notation from
remark 3.1) are geometrically optimal - this follows from proposition 3.1
and the properties of Martinet sub-Lorentzian structures described in \cite%
{g8}.

\bigskip

\begin{acknowledgement}
This work was partially supported by the Polish Ministry of Research and
Higher Education, grant NN201 607540.
\end{acknowledgement}

\textsl{Faculty of Mathematics and Science, Cardinal Stefan Wyszy\'{n}ski
University, ul. Dewajtis 5, 01-815 Waszawa, Poland}

\end{document}